\documentclass[12pt, reqno]{amsart}
\usepackage{amsmath,amsfonts,amssymb,amsthm,enumerate,appendix,caption}
\usepackage{cmap,mathtools}
\usepackage{paralist,bm}
\usepackage{graphics} 
\usepackage{epsfig} 
\usepackage{graphicx}
\usepackage{braket}
\usepackage{epstopdf}
 \usepackage[colorlinks=true]{hyperref}
\usepackage{xcolor}
\definecolor{myblue}{rgb}{0.0, 0.0, 1.0}
\definecolor{mygreen}{rgb}{0.01,0.75,0.20}
\hypersetup{ linkcolor=myblue, colorlinks=true, urlcolor=black, citecolor=red}
\usepackage{hyperref}

\newtheorem{theorem}{Theorem}[section]
\newtheorem{corollary}[theorem]{Corollary}

\newtheorem{proposition}[theorem]{Proposition}

\theoremstyle{definition}
\newtheorem{definition}[theorem]{Definition}

\newtheorem{remark}[theorem]{Remark}

\newtheorem{example}[theorem]{Example}
\newtheorem{application}[theorem]{Application}

\theoremstyle{definition}

\usepackage[margin=1in]{geometry}

\numberwithin{equation}{section}

\newcommand{\dy}{\,\mathrm{d}y}

\def\ga{\alpha}     \def\gb{\beta}       \def\gg{\gamma}
       \def\gd{\delta}     
                         \def\vge{\varepsilon}
\def\gf{\varphi}       \def\vgf{\varphi}    
            \def\gl{\lambda}
\def\gm{\mu}                 
            
\def\gs{\sigma}       
      \def\gw{\omega}
                
\def\Gg{\Gamma}     \def\Gd{\Delta}

\def\Gw{\Omega}              

\newcommand{\eat}[1]{}

\DeclarePairedDelimiter\norm{\lVert}{\rVert}%
\makeatletter
\let\oldnorm\norm
\def\norm{\@ifstar{\oldnorm}{\oldnorm*}}

\newcommand{\om} {\omega}
\newcommand{\Om} {\Omega}

\newcommand{\la} {\lambda}

\newcommand\restr[2]{{
  \left.\kern-\nulldelimiterspace 
  #1 
  \right|_{#2} 
  }}

\def\w{{\widetilde w}}

\def\w2{{W^{1,2}_0(\Om)}}
\def\hh2{{H^1_0(\Om)}}

\def\C{{\mathcal C}}

\def\N{{\mathbb N}}
\def\F{{\mathcal F}}

\def\R{{\mathbb R}}

\def\({{\Big(}}
\def\){{\Big)}}

\def\ws2{{\F_{\frac{N}{2}}}}

\def\c1{{\C_c^1}}

\def\dy{{\rm d}y}

\newcommand\supp{\mathrm{supp}\,}

\newcommand{\Hmm}[1]{\leavevmode{\marginpar{\tiny%
			$\hbox to 0mm{\hspace*{-0.5mm}$\leftarrow$\hss}%
			\vcenter{\vrule depth 0.1mm height 0.1mm width \the\marginparwidth}%
			\hbox to
			0mm{\hss$\rightarrow$\hspace*{-0.5mm}}$\\\relax\raggedright #1}}}
\newcommand{\loc}{{\rm loc}}

\begin{document}
	\title[The Landis conjecture]{The Landis conjecture via Liouville comparison principle and criticality theory}
	
	
	\author {Ujjal Das}
	
	\address {Ujjal Das, Department of Mathematics, Technion - Israel Institute of
		Technology,   Haifa, Israel}
	
	\email {ujjaldas@campus.technion.ac.il, getujjaldas@gmail.com}
	\author{Yehuda Pinchover}
	\address{Yehuda Pinchover,
		Department of Mathematics, Technion - Israel Institute of
		Technology,   Haifa, Israel}
	\email{pincho@technion.ac.il}
	
	\begin{abstract}
We give partial affirmative answers to Landis conjecture in all dimensions for two different types of  linear, second order, elliptic operators in a domain $\Omega\subset \mathbb{R}^N$. In particular, we provide a sharp decay criterion that ensures when a solution of a nonnegative Schr\"odinger equation in $\mathbb{R}^N$ with a potential $V\leq 1$ is trivial. Moreover, we address the analogue of Landis conjecture for quasilinear problems. Our approach relies on the application of Liouville comparison principles and criticality theory.

		\medskip
		\noindent  2020  \! {\em Mathematics  Subject  Classification.}
		Primary 35J10; Secondary 35B09, 35B53, 35B60.\\[-3mm]

	\end{abstract}
\maketitle
\section{Introduction}
Let $V$ be a bounded function defined on a domain $\Gw$ in $\R^N$, $N\in \N$, where $\Om$ is either $\R^N$ or an exterior domain, i.e., an unbounded domain with a nonempty compact boundary. The Landis conjecture \cite{Landis} claims that if a bounded function $u$ 
solves the Schr\"odinger equation
\begin{equation}\label{eq_Sch}
H [\varphi] :=(-\Gd  + V)[\varphi] = 0 \qquad \mbox{in } \Gw,
\end{equation}
and decays faster than $\mathrm{e}^{-k|x|}$, for some
$k>  \sqrt{\|V\|_{L^{\infty}(\Om)}}$ , then $u = 0$. Certainly, the conjecture is sharp in dimension one, since a potential $V$ which equals to a constant $c\in\R$ outside a compact set, admits an exponential decaying solution only if $c=k^2>0$ and in this case such a solution decays precisely as $\mathrm{e}^{-|k||x|}$. On the other hand,
 Meshkov \cite{Meshkov} disproved the conjecture for a complex-valued potential $V$, by constructing a nontrivial, complex-valued, bounded solution $u$ of \eqref{eq_Sch} in $\Gw=\R^2$ which decays as $\mathrm{e}^{-c|x|^{{4}/{3}}}$ for some $c>0$. Moreover, Meshkov showed that the exponent $4/3$ is sharp in the sense that if $u = o(\mathrm{e}^{-|x|^{{4}/{3}+\vge}})$ for some $\vge>0$ as $|x| \rightarrow \infty$, then $u = 0$. Although, due to Meshkov's result, the Landis conjecture is settled for complex-valued bounded potentials $V$ and complex-valued bounded solutions $u$, it still remains open in the real-valued case.  The study of Landis conjecture for real-valued case started gaining attention since the work of Bourgain and Kenig \cite{BK}, where using a {\it{Carleman-type inequality}} the authors  derived that  for bounded normalized solution $u$ with $|u(0)| = 1$
 \begin{align} \label{Eq:Quan_Landis}
 \inf_{|x|=R} \|u\|_{L^{\infty}(B_1(x))} \geq c \mathrm{e}^{-cR^{{4}/{3}}\log R} \, \ \text{for some} \ c>0 \ \text{and} \ R >> 1 \,,
 \end{align}
  where $c$ is independent of $R$.
Consequently, if a solution $u$ of \eqref{eq_Sch} decays as   $o(\mathrm{e}^{-c|x|^{{4}/{3}}\log |x|})$, then $u = 0$. Note that this improves Meshkov's result in the real-valued case by allowing a slower decay of $u$ than that of Meshkov (which is precisely, $\mathrm{e}^{-|x|^{{4}/{3}+\vge}}$) for the validity of Landis conjecture; yet, this decay is faster than what Landis originally proposed. 

This led to the study of a weaker version of Landis conjecture in the real-valued case with $V\in L^\infty(\R^N)$, i.e., to examine whether  the condition $u=o(\mathrm{e}^{-|x|^{1+\varepsilon}})$ implies $u = 0$ \cite[Question 1]{Kenig}.  
In $\R^2$, for bounded potential $V \geq 0$ (which of course implies that $-\Gd  + V\geq 0$), Kenig-Silvestre-Wang \cite{KSW} improved the quantitative estimate  \eqref{Eq:Quan_Landis} with $\mathrm{e}^{-cR\log R}$ in the right-hand side of \eqref{Eq:Quan_Landis}, and hence, proved the Landis conjecture with a further slower decay rate than that of Bourgain-Kenig. In order to derive this improved quantitative estimate, instead of Carleman-type inequality, they exploited the nonnegativity of $H$, and in particular, they used a {\it{three-ball inequality}} derived from the {\it{Hadamard three-circle theorem}}. Note that they considered a more general operator than $H$ by allowing a drift term. Utilizing the core idea in \cite{KSW} with appropriate modifications, a series of results on the weak Landis conjecture in $\R^2$ have been obtained for more general second-order elliptic operators with a bounded potential $V$ which is allowed to be sign-changing, for instance, see \cite{Davey3,Davey1,Davey2, KW_MRL} and the references therein. Another approach can be found in \cite{Balch,Davey6,LMNN}, where the  authors improved the quantitative estimate \eqref{Eq:Quan_Landis} when $\Om=\R^2$ using quasiconformal mappings and the nodal structure of the solution $u$. However, both of the above two techniques are difficult to extend in higher dimensions.  In fact, there are only a handful of results that address the Landis conjecture in its original form in higher dimensions (cf. \cite{ABG,Rossi,Sirakov}).

More precisely, in \cite{Rossi}, Rossi used ODE techniques to treat the one-dimensional case and used it to prove the Landis conjecture for radially symmetric, second order, linear, uniformly elliptic operator $\mathcal{L}$ in nondivergence form with bounded coefficients, and with a bounded potential $V(x)=\tilde V(|x|)$ in an exterior domain in $\R^N$.  In addition, using two different approaches: one by probabilistic tools \cite{ABG} and another by the comparison principle \cite{Sirakov}, the Landis conjecture in higher dimensions is proved for a general second-order uniformly elliptic operator, where in \cite{ABG} it is assumed that the coefficients are bounded, while in \cite{Sirakov} this assumption is relaxed.  The obtained decay rates under which the Landis conjecture is proved in these articles depend on the coefficients of the operator and the uniform ellipticity constant, and they are not sharp in general. In this context, it is important to mention that the assumptions in both articles \cite{ABG,Sirakov} ensure that  the generalized principal eigenvalue of the associated operator $\mathcal{L}$ is nonnegative, namely, $\mathcal{L} \geq 0$ (see Definition~\ref{def_nonneg}).

We also refer the reader to Landis-type results for the time-dependent Schr\"odinger equation and the heat equation \cite{EKPV10,EKPV16}. For developments in the discrete setting see \cite{DKP,FRS,FV,JLMP,LM} and the references therein.

\medskip 

 In this article, we obtain new partial affirmative answers to Landis conjecture in all dimensions for two different types of linear elliptic operators in a domain $\Gw\subset \R^N$. Moreover, we address the analogue of Landis conjecture for quasilinear problem as well, see Section~\ref{Sec-Quasi}. The novelty of our approach to proving the Landis conjecture relies on the application of Liouville comparison principles (Theorem~\ref{LC}) for {\em nonnegative} Schr\"odinger operators, and criticality theory for general {\em nonnegative} second order elliptic operators of the divergence form \eqref{Eq:Ellp_op}.  
 
 The first illustration of our result is the case of a nonnegative Schr\"odinger operator 
 \begin{align} \label{Eq:Def_Schordinger}
 H:=-\Gd +V \ \ \mbox{in} \ \ \R^N \,,
\end{align}  where $N \geq 1$, $V\leq 1$, and $V\in L^{q}_\loc(\R^N)$, ${q}>N/2$ if $N\geq 2$ and $q=1$ if $N=1$. Note that we allow the potential $V$ to be unbounded from below. 
 \begin{theorem}\label{Thm:Schrodinger}
Let $H$ be  a nonnegative Schr\"odinger operator in $\R^N$ as given in \eqref{Eq:Def_Schordinger}. 
If $u \in W^{1,2}_{\loc}(\R^N)$ is a solution of the equation $H[\varphi]=0$ in $\R^N$ satisfying   
\begin{equation}\label{Eq:cond_Thm:Schrodinger}	
|u(x)| = \begin{cases}
O(1) & N=1,\\
O\left(|x|^{(2-N)/2}\right)  &N\geq2,
\end{cases} 
\ \mbox{ as } |x|\to\infty, \ \mbox{ and }\  \liminf_{|x|\to \infty} \frac{|u(x)||x|^{(N-1)/2}}{\mathrm{e}^{-|x|}} =0 \,,
\end{equation}
then $u=0$.   
\end{theorem}
We emphasize that the above result not only proves the Landis conjecture in any dimension
for nonnegative $H$ with potential $V$ which can be unbounded from below, but also the theorem assumes
a slower decay rate of $u$ than the one Landis originally proposed in his conjecture. It is noteworthy that the theorem does not necessitate the exponential decay of $u$ at infinity in $\R^N$; it is enough to have a sequence $(x_n)$ with $|x_n| \to\infty$ such that $(u(x_n))$ decays faster than $\mathrm{e}^{-|x_n|}/|x_n|^{\frac{N-1}{2}}$. Furthermore, our result is sharp, as the following example demonstrates. 
\begin{example}\label{ex1}
 Let $\mathcal{G}_{H_1}$ be a minimal positive Green function of the operator $H_1:= -\Gd+1$ in $\R^N$ having a singularity at the origin. It is well known (see the Appendix and for more details \cite{Norio_Shimakura}) that 
$\mathcal{G}_{H_1}(x) \asymp 
|x|^{(1-N)/2}\,\mathrm{e}^{-|x|}$ for  $|x|\geq 1 $. 
Let $0\lneq W\in C^\infty_c(\R^N)$ and consider the {\em generalized principal eigenvalue} of $H_1$
$$\gl_0=\gl_0(H_1,W,\R^N):=\sup\{\gl\in\R \mid H_1-\gl W\geq 0 \ \mbox{ on }  C^\infty_c(\R^N)\}. $$
Then $(H_1-\gl_0 W)[\varphi]=0$ in $\R^N$ admits a positive solution which is an Agmon ground state $\Psi$.  It follows that $\Psi(x) \asymp \mathcal{G}_{H_1}(x) \asymp |x|^{(1-N)/2}\mathrm{e}^{-|x|}$ for $|x|\geq 1$.
\end{example}
%
If one looks a little deeper into Theorem \ref{Thm:Schrodinger} to understand how the functions $|x|^{(1-N)/2}\mathrm{e}^{-|x|}$ and $|x|^{(2-N)/2}$ in \eqref{Eq:cond_Thm:Schrodinger} are related to the operator $H$, it can be realized that the asymptotics at infinity of these functions are same  as positive solutions of minimal growth at infinity (see Definition~\ref{def:minimalgrowth}) of the equations $(-\Delta +1)[\varphi]=0$, and $(-\Delta -W_{\mathrm{\mathrm{opt}}})[\varphi]=0$, respectively, where 
$W_{\mathrm{opt}}$ is an optimal Hardy-weight of $-\Delta$ in $\R^N$, in the sense of \cite{DFP}. Precisely, $W_{\mathrm{opt}}(x)\asymp (N-2)^2|x|^{-2}/4$ for $N\geq 3$ (in the subcritical case) and  $W_{\mathrm{opt}}=0$ for $N=1,2$ (in the critical case).  

This crucial observation as mentioned above, leads us to extend Theorem \ref{Thm:Schrodinger} to nonnegative  Schr\"odinger-type operators (i.e., with variable coefficients) in a general domain $\Om$ in $\R^N$. 
Let $A_k:\Omega
	\rightarrow \mathbb{R}^{N^2}$; $k=1,2$ be two symmetric  matrix valued functions such that 	{$$ C_K^{-1} I_N\le A_k\le C_K I_N \quad \mbox{in any} \ K\Subset \Om  $$  for some $C_K>0$ depending on $K$.  Consider the following two operators on $\Om$	 
	\begin{equation}\label{eq_Hk}
	H_k :=P_k+V_k= -\mathrm{div}(A_k\nabla ) +V_k\, , \qquad  k=1,2,
	\end{equation}
		where $V_k\in L^q_\loc(\Om)$ with $q>N/2$ if $N\geq 2$ and $q=1$ if $N=1$.
	Then we have the following answer to Landis conjecture (we will call such a result {\em Landis-type theorem}):
\begin{theorem}\label{Thm:Gen_Schrodinger} 
Let $\Om$ be a domain in $\R^N$, $N\geq 1$, and	 consider two nonnegative  Schr\"odinger-type operators $H_k$ on $\Om$ as in \eqref{eq_Hk}.  Let $W \in L^{q}_{\loc}(\Om)$ ($q$ as above) satisfies $W \geq V_2$ outside a compact set in $\Om$.
  
	Suppose that $H_1$ is critical in $\Gw$ with an Agmon ground state $\Psi_{H_1}$. Let  $u \in W^{1,2}_{\loc}(\Om)$ be a solution of the equation $H_2 [\varphi]=0$ in $\Om$  satisfying   
\begin{equation}	
|u|^2A_2\leq C\Psi_{H_1}^2A_1 \  \mbox{ in } \Gw \ \ \mbox{ and } \ \  \liminf_{x\to \bar{\infty}} \frac{|u(x)|}{G_{P_2+W}(x)} =0 \nonumber\,,
\end{equation}
where $G_{P_2+W}$ is a positive solution of minimal growth of $(P_2 + W)[\varphi]=0$ in a neighbourhood of infinity in $\Om$, and $\bar \infty$ is the ideal point in the one-point compactification of $\Gw$.
Then $u = 0$.   
\end{theorem}

\medskip

The above Landis-type theorem for symmetric operators leads us to extensions of such a result to the nonsymmetric case of linear second-order elliptic operators with real coefficients of the following divergence form:
\begin{align} \label{Eq:Ellp_op}
	\mathcal{L}[\varphi]:= -\mathrm{div} \left[\big(A(x)\nabla \varphi +  \varphi \tilde{b}(x) \big) \right]  +
	b(x)\cdot\nabla \varphi   +c(x) \varphi, \qquad x\in \Gw, 
\end{align}
where $A(x):=(a_{ij}(x))_{N \times N}$ is a symmetric, locally uniformly elliptic matrix, $a_{ij}\in L^\infty_\loc(\Gw)$, $b_i,\,\tilde{b_i} \in L^{\tilde q}_{\mathrm{loc}}(\Omega)$, $\tilde q > N$, $c \in L^{q}_{\mathrm{loc}}(\Omega)$, $ q >N/2$, and $\Gw\subset \R^N$ is a domain. We will consider {\em weak (super)solutions} of $\mathcal{L}[\varphi]=0$ (see Definition~\ref{def_weak}). We say that $\mathcal{L}$ is {\em nonnegative in} $\Gw$ (in short $\mathcal{L}\geq 0$) if the equation $\mathcal{L}[\varphi]=0$ admits a positive (super)solution in $\Gw$. 

In general, Landis conjecture may not hold for such an operator,  see for example \cite[Theorem~2]{Plis}. Nevertheless, under the above restrictions on the coefficients and the crucial assumption $\mathcal{L}\geq 0$, we establish the following Landis-type theorem. This result is in the spirit of the refined maximum principle in \cite[Section~4]{P99}. 
\begin{theorem}\label{thm_element}
Given an elliptic operator $\mathcal{L}$ of the form \eqref{Eq:Ellp_op} with coefficients satisfying the above regularity conditions, we assume that $\mathcal{L}\geq 0$ in $\Gw$. Let $u \in W^{1,2}_{\loc}(\Om)$ be a solution of the equation
$\mathcal{L}[\varphi]=0$ in $\Gw$ such that $|u|=O(G_{\mathcal{L}})$ as $x\to \bar{\infty}$, where $G_{\mathcal{L}}$ is a positive solution of the equation
$\mathcal{L}[\varphi]=0$ of minimal growth at infinity in $\Gw$. Then either $u=0$, or $\mathcal{L}$ is critical in $\Gw$ and  $u$ is (up to a multiplicative constant) the Agmon ground state of $\mathcal{L}$.   
\end{theorem}
We use the {\it{maximal $\vge$-argument}} to prove the above result.  Observe that, if we additionally assume $\liminf_{x \rightarrow \bar{\infty}} (|u(x)|/G_{\mathcal{L}}(x))=0$ in the above theorem, then the second possible assertion cannot occur. Hence, it follows that $u =0$ in $\Om$.

In a similar spirit as in the preceding theorem, we prove the following variant of Landis-type theorem associated to the operator $\mathcal{L}+V$ under certain conditions on $V$.
\begin{theorem}\label{thm_element1}
	Consider a nonnegative elliptic operator  $\mathcal{L}+V$ in $\Gw$, where  $\mathcal{L}$ is of of the form \eqref{Eq:Ellp_op} with coefficients satisfying the aforementioned regularity conditions and $V\in L^{q}_{\mathrm{loc}}(\Omega)$ for some $q > N/2$. Let $0\leq W \in L^{q}_{\mathrm{loc}}(\Omega)$, $q>N/2$, be a critical Hardy-weight for $\mathcal{L}$ with Agmon ground state $\Psi_{\mathcal{L}-W}$, and suppose that  $V+W\geq 0$ in $\Gw$. 
	
	Let $u$ be a solution of the equation
	$(\mathcal{L}+V)[\varphi]=0$ in $\Gw$ such that 
	$$|u|=O(\Psi_{\mathcal{L}-W}) \ \mbox{ as} \ \ x\to \bar{\infty} \ \  \mbox{ and }  \ \ \liminf_{x\to \bar{\infty}} \frac{|u(x)|}{\Psi_{\mathcal{L}-W}(x)} =0.$$
	 Then $u=0$.
\end{theorem}
In comparison with the Landis-type theorems in \cite{ABG}, we see that the above two theorems relax the boundedness assumption of the coefficients of the elliptic operators and the restriction on the domain. In addition, while comparing with \cite[Theorem 1.2]{Sirakov}, one should note that we relax the boundedness assumption on $(a_{ij})$. It is remarkable that the theorems \ref{thm_element} and \ref{thm_element1} provide a slower (and in some cases even optimal) decay of $u$ than the one  prescribed in \cite{ABG,Sirakov} for Landis-type theorem to hold for certain operators $\mathcal{L}$, see Application \ref{app-4_3}.

It follows from the above two theorems that to conclude $u=0$, it suffices to find good lower bounds for $G_{\mathcal{L}}$ and $\Psi_{\mathcal{L}-W}$, respectively. Several possible ways to get such estimates for these functions are presented in sections~\ref{sec_3} and \ref{sec_general}. It is natural to anticipate that such lower  estimates  may help us to prove Landis conjecture under a slower decay assumption on $u$ than the decay Landis proposed. We show in sections~\ref{sec_3} and \ref{sec_general} that this is indeed possible for some $\mathcal{L}$.
In fact, we highlight that under certain restrictions, the Landis-type theorem holds on $\R^N$ even under the assumption that $u$ decays only polynomially, i.e., $u(x)=O(1/|P(|x|)|)$ as $x\to \infty$, where $P$ is a polynomial, see for example  Application~\ref{app_elem}.  

The paper is structured as follows. In the following section, basic notions and some results from criticality theory are stated. In Section~\ref{sec_3}, we prove the Landis-type theorems for Schr\"odinger-type operators (theorems~\ref{Thm:Gen_Schrodinger} and \ref{Thm:Schrodinger}), and present various applications of them. Section~\ref{sec_general} is dedicated to the proofs of theorems~\ref{thm_element} and \ref{thm_element1} which are extensions of the Landis-type theorem for Schr\"odinger operators (Theorem~\ref{Thm:Gen_Schrodinger}) to the case of nonselfadjoint operators. This section also covers several applications of these theorems. Section~\ref{Sec-Quasi} deals with the quasilinear case. Using a Liouville comparison principle for the $(p,A)$-Laplace operator with a potential \cite{PP,PTT}, we prove a Landis-type theorem for such operators.  Finally, in Section~\ref{sec-6}, we discuss the limitation of our results, namely, our crucial assumption that the operators considered throughout the paper are nonnegative.
\section{Preliminaries}\label{sec-prelim}
This section is devoted to some basic notions and important results in criticality theory that are essential for this article. In particular, we recall a Liouville comparison principle for Schr\"odinger-type operators.

We use the following notation and conventions:

\begin{itemize}
\item All the functions considered in this article are real-valued measurable functions.
\item     $\mathbb{S}^{N-1}$ denotes the unit sphere in $\R^N$.
	\item For two subsets $\om_1,\om_2$, we write $\om_1 \Subset \om_2$ if  $\overline{\om_1}$ is a compact set, and $\overline{\om_1}\subset \om_2$.
	\item  For two positive measurable functions defined in $\gw \subset\R^N$, we write  $g_1\asymp g_2$ in
	$\gw$ if 	$C^{-1}g_{2}(x)\leq g_{1}(x) \leq Cg_{2}(x)$ for a.e.  $x\in \gw$, where $C$ is a positive constant.	
	\item For any real-valued measurable function $f$, we denote
	$f^+:=\max\{0,f\}$.
	\item For a subspace $X(\om)$ of measurable functions on $\om$, 
	$X_c(\om) :=\{f \!\in \! X(\om)\!\mid\! \supp f \!\Subset\! \om\}$.  
	\item The matrix inequality $A\leq B$ means that $B-A$ is a nonnegative definite matrix.
	\item By $I_N$ we denote the $N$-dimensional identity matrix.
	\item For a domain $\Om$ in $\R^N$, we denote the distance function to the boundary $\partial \Om$ by $\mathrm{d}_{\Om}$.
	\item Let $\bar \infty$  be an ideal point, throughout the paper, we consider the one-point compactification $\hat\Gw=\Gw\cup\{\bar \infty\}$ of the domain $\Gw$. So,  if $(x_n)\subset \Om$, then  $x_n \rightarrow \bar \infty$ if and only if for any $K \Subset \Om$ there exist $N_K \in \N$ such that $x_n \in \Om \setminus K$ for all $n \geq N_K$. 
	\item For two functions $ f,g :\Om \to \R $,  we write $f = O(g)$ as $x \rightarrow \bar{\infty}$  if 
\begin{align*}
	  \limsup_{x\to \bar\infty}\frac{|f(x)|}{|g(x)|}<\infty.
\end{align*}
If the above limsup is zero, then we simply write $f = o(g)$.
\end{itemize}

In this article we consider second order elliptic operators $\mathcal{L}$ in the divergence form \eqref{Eq:Ellp_op}. If $\tilde b_i=b_i =0$ for all $1\leq i \leq N$, then we denote such an operator by $H$ and call it a {\em Schr\"odinger-type operator}. The (sub/super) solution of the equation $\mathcal{L}[\varphi]=0$ in $\Om$ is considered in the weak sense as defined below.  
\begin{definition}[Weak (sub/super) solution] \label{def_weak} 
	Given an  elliptic operator $\mathcal L$ of the form \eqref{Eq:Ellp_op},  we say that $u\in W^{1,2}_{\loc}(\Gw)$ is a  {\em (weak) solution} (resp., {\em supersolution)} of the  equation  $\mathcal{L}[\vgf]=0$ in $\Gw$, if for any (resp., nonnegative) $\phi \in C_c^{\infty}( \Gw )$   we have
	\begin{equation*}\label{weak_solution}
		\int_{\Omega}\left[\sum_{i=1}^N\left((\sum_{j=1}^N a^{ij}\partial_j u)+u \tilde{b}^i \right)\partial_i \phi+(\sum_{i=1}^Nb^i \partial_i u+cu )\phi \right] \!\mathrm{d}x  = 0\ \ 
	(\text{resp.,} \geq 0).
	\end{equation*} 
	In this case we write $\mathcal L[u]=0$ (resp., $\mathcal L [u] \geq 0$) in $\Om$.  Furthermore, $u$ is a (weak) {\em subsolution} of $\mathcal L[\varphi]=0$  in $\Gw$ if $-u$ is a supersolution of this equation in $\Gw$.
\end{definition}

\begin{remark} \label{Rmk:subsol}
The positive part $u^+$ of a (sub)solution $u$ of the equation $\mathcal{L}[\varphi]=0$ in $\Om$ is a subsolution of the same equation \cite[Lemma~2.7]{Ag}. Also, by the {\it{weak Harnack inequality}}, the {\it{strong maximum principle}}  holds true, i.e., any nontrivial, nonnegative supersolution is a.e. strictly positive \cite[Theorem~8.18]{GT}. Moreover, such supersolutions are locally bounded. 
\end{remark}
\begin{definition}[{Minimal growth \cite{Ag}}]\label{def:minimalgrowth}
	Let $\mathcal{L}$ be an elliptic operator of the form \eqref{Eq:Ellp_op}. A function $u$ is said to be a {\em positive solution of the equation $\mathcal{L}[\varphi]=0$ of minimal growth at $\bar \infty$}  if for some $K\Subset\Gw$,  $u \in W^{1,2}_{\loc}(\Om\setminus K)$, $u$  is a positive solution of the equation $\mathcal{L}[\varphi]=0$ in $\Gw\setminus K$, and for any  positive supersolution $v$ of $\mathcal{L}[\varphi]=0$ in a subdomain $\Gw\setminus K_1$ with $K\Subset K_1\Subset \Gw$ such that   $K_1$ is a Lipschitz bounded domain in $\Om$,  the inequality $u \leq v$ on $\partial K_1$ implies  $u \leq v$ in $\Gw\setminus K_1$. We will denote such a solution by $G_{\mathcal{L}}$.
	
	Fix $y\in \Gw$. A function $\mathcal{G}_{\mathcal{L}}(\cdot,y) \in W^{1,2}_\loc(\Gw\setminus\{y\})\cap L^1_\loc(\Gw)$ is called a {\it{minimal positive Green function}} of $\mathcal{L}$ in $\Gw$ with singularity at $y$, if $\mathcal{G}_{\mathcal{L}}(\cdot,y)$ is a positive solution of the equation $\mathcal{L}[\varphi]=0$ in $\Gw\setminus\{y\}$ which has a minimal growth at $\bar\infty$ in $\Gw$, and satisfies the equation $\mathcal{L}[\mathcal{G}_{\mathcal{L}}](\cdot,y)=\gd_y$ in the distributional sense, where $\gd_y$ is the Dirac measure with a charge at $y$. With a slight abuse of notation, we will write the function $\mathcal{G}_{\mathcal{L}}(\cdot,y)$ as $\mathcal{G}_{\mathcal{L}}(\cdot)$ with singularity at $y$. We may skip specifying the singularity $y$ if $y$ is a generic point in $\Gw$. 
	
	A positive solution of $\mathcal{L}[\vgf]=0$ in $\Gw$ which has a minimal growth at $\bar\infty$ is called an {\em Agmon ground state} of $\mathcal{L}$ in $\Gw$, and it will be denoted by  $\Psi_{\mathcal{L}}$.
\end{definition}
\begin{definition}[Nonnegativity of operators]\label{def_nonneg}
An elliptic operator $\mathcal{L}$ of the form \eqref{Eq:Ellp_op} is {\em nonnegative in} $\Gw$ (in short, $\mathcal{L} \geq 0$) if the equation $\mathcal{L}[\vgf]=0$ admits a positive (super)solution in $\Gw$.

	Let  $\mathcal L\geq 0$ in $\Gw$, and $0\leq w \in L^{q}_\loc(\Om)$, $q>N/2$. The {\em generalized (weighted) principal eigenvalue}  is given by
	$$\gl_0(\mathcal{L},w,\Gw):= \sup\{\gl \mid \mathcal{L}-\gl w\geq 0 \ \mbox{ in } \Gw\}.$$
\end{definition}
\begin{definition}[(Sub)critical operators]
A nonnegative operator $\mathcal{L}$ of the form \eqref{Eq:Ellp_op} is said to be {\em subcritical} in $\Gw$ if there exists $W\gneq 0$, $W\in L_\loc^{q}(\Gw)$, with $q>N/2$ such that   $\mathcal{L}-W\geq0$ in $\Gw$. Such a function $W$ is called a {\em Hardy-weight} for $\mathcal{L}$ in $\Gw$.  The operator $\mathcal{L}$ is {\em critical} in $\Gw$ if $\mathcal{L}\geq 0$ in $\Gw$, but $\mathcal{L}$ is not subcritical in $\Gw$.
\end{definition}

\begin{definition}[Critical Hardy-weights]
A Hardy-weight $W$ of an operator  $\mathcal{L}$ of the form \eqref{Eq:Ellp_op} is said to be a {\em critical Hardy-weight} for $\mathcal{L}$  in $\Gw$ if $\mathcal{L}-W$ is critical in $\Gw$.
\end{definition}
\begin{remark}
	For some possible ways to produce critical Hardy-weights, see \cite{DFP,Murata1,P_JDE1}.
\end{remark}

\begin{theorem}[Characterization of (sub)criticality \cite{P07}] \label{Thm:char_cri}
Let $\mathcal{L}\geq 0$ be an elliptic operator of the form \eqref{Eq:Ellp_op}. The following assertions are equivalent:
	\begin{enumerate}
\item $\mathcal{L}$ is  subcritical (resp., critical) in $\Gw$.
\item $\mathcal{L}$ admits (resp., does not admit) a minimal positive Green function with singularity at some  $y\in\Gw$.  
\item $\mathcal{L}$ admits (resp., does not admit) two linearly independent positive supersolutions.   
\item $\mathcal{L}$ admits (resp., does not admit) a positive supersolution which is not a solution.
\item $\mathcal{L}$ does not admit (resp., admits) an Agmon ground state.
\end{enumerate}
\end{theorem}	

\begin{theorem}[Liouville comparison principle {\cite[Theorem 1.7]{P}}]\label{LC} Let $\Omega$ be a domain in $\R^N$, $N\geq 1$.
	Consider two Schr\"odinger-type operators defined on $\Omega$ of the
	form
	\begin{equation*}
	H_k:=P_k+V_k=-\nabla\cdot(A_k\nabla)+V_k\, ,   \qquad k=1,2,
	\end{equation*}
	such that  $V_k\in L^{q}_{\mathrm{loc}}(\Omega)$ for
	some $q>{N}/{2}$, and 
	 $A_k:\Omega
	\rightarrow \mathbb{R}^{N^2}$ is a symmetric  matrix
	valued function such that for every compact set $K\subset \Omega$
	there exists $\mu_K>1$ satisfying
	\begin{equation*} 
	\mu_K^{-1}I_N\le A_k(x)\le \mu_K I_N \qquad \forall x\in K.
	\end{equation*}
	Assume that the following assumptions hold true.
	\begin{itemize}
		\item[(i)] The operator  $H_1$ is critical in $\Omega$ with its Agmon ground state $\Psi_{H_1}$.
		
		\item[(ii)]  $H_2\geq 0$ in $\Om$, and there exists $\Phi_2 \in W^{1,2}_{\mathrm{loc}}(\Omega)$ such that
		$H_2 \Phi_2 \leq 0$ in $\Omega$, and  $\Phi_2^+\neq 0$.
		
		\item[(iii)] The following matrix inequality holds
		\begin{equation*}
		(\Phi_2^+)^2(x) A_2(x)\leq C\Psi_{H_1}^2(x) A_1(x) \quad  \mbox{ a.e.
			in } \Omega,
		\end{equation*}
		where $C$ is a positive constant.
	\end{itemize}
	Then the operator $H_2$ is critical in $\Omega$, and $\Phi_2$ is its Agmon
	ground state.    
\end{theorem}
\begin{remark}
For general second order elliptic operators, a variant of the above Liouville comparison principle (for {\it{strong solutions}}) is given in  \cite[Theorem 2.2]{ABG} under some regularity assumptions and additional conditions.   
\end{remark}

\section{Landis-type theorems for Schr\"odinger-type operators}\label{sec_3}
In this section we prove Theorem~\ref{Thm:Gen_Schrodinger} which is a generalized version of Theorem~\ref{Thm:Schrodinger} and present various applications of Theorem~\ref{Thm:Gen_Schrodinger} to certain particular Schr\"odinger-type equations.
\begin{proof}[Proof of Theorem \ref{Thm:Gen_Schrodinger}] 
	Suppose that $u\neq 0$. Without loss of generality, we may assume that  $u^+\neq 0$.  By the Liouville comparison principle, Theorem \ref{LC}, with the operators $H_1$ and $H_2$, it follows that $H_2$ is critical and $u=u^+>0$ is its Agmon ground state.

Now, since $u>0$ and $W\geq V_2$ outside a compact set in $\Om$, it follows that 
	$$(P_2+W)[u]=H_2[u] + (W-V_2)u = (W-V_2)u \geq 0 $$
outside a compact set in $\Om$. 
	So, $u$ is a positive supersolution of the equation $(P_2+W)[\varphi]=0$ outside a compact set in $\Om$. As $G_{P_2+W}$ is a positive solution of minimal growth at infinity in $\Gw$ of the same equation, we conclude that $u \geq C G_{P_2+W}$ near infinity in $\Om$ for some $C>0$. This contradicts our assumption 
$$\liminf_{x\to \bar{\infty}} \frac{|u(x)|}{G_{P_2+W}(x)} =0.$$
Therefore, $u = 0$.
\end{proof}
To illustrate Theorem \ref{Thm:Gen_Schrodinger}, we provide some examples of Schr\"odinger-type operators where we can explicitly estimate from below the associated positive solutions of minimal growth at infinity and use it to prove Landis-type theorems for the corresponding Schr\"odinger equation.
\begin{application}[Polynomial decay case in $\R^N$]\label{app_elem}
 As an elementary example of Theorem \ref{Thm:Gen_Schrodinger}, consider $\Om=\R^N$, $A_k=I_N$ for $k=1,2$,  $V_1=-W_{\mathrm{cri}}$,  and $V_2\leq W =0$ outside a compact set, where $W_{\mathrm{cri}} \geq 0$ is a critical Hardy-weight of $-\Delta$ in $\R^N$. In this case, $P_1=P_2=-\Delta$. Clearly, the associated positive solutions of minimal growth at infinity of the respective equations satisfy the following properties for some $C>0$: $$ CG_{-\Delta}\leq G_{P_1-W_{\mathrm{cri}}} \asymp \Psi_{H_1} \quad \mbox{as} \ |x| \rightarrow \infty\,, $$ and 
 \begin{align*}
 G_{P_2+W}(x)\geq CG_{-\Delta}(x) \asymp \begin{cases} 1 \ \ \qquad \ \mbox{if} \ \ N=1,2, \\
 |x|^{2-N} \ \ \mbox{if} \ \ N \geq 3 \,,
 \end{cases} \mbox{ as } |x| \rightarrow \infty,
 \end{align*}
 see \cite{PTT}.
Thus, in dimension one and two, if $u \in W^{1,2}_{\loc}(\R^N)$ is a bounded solution of $(-\Delta +V_2)[\varphi]=0$ in $\R^N$ satisfying $\liminf_{|x|\rightarrow \infty} |u(x)|=0$, then Theorem \ref{Thm:Gen_Schrodinger} ensures that $u = 0$. Whereas, in dimension $N \geq 3$, if $u$ satisfies the polynomial decay $|u|=O(|x|^{2-N})$, or even $|u|=O(|x|^{(2-N)/2})$ (which corresponds to the asymptotic of the Agmon ground state  $\Psi_{H_1}$ of the critical Hardy operator $H_1=-\Gd-W_{\mathrm{cri}}\, $, where $W_{\mathrm{cri}}=  \frac{(N-2)^2}{4}|x|^{-2}$ near infinity (see, \cite[Theorem~4.12]{DFP})), and  
$$\liminf_{|x|\rightarrow \infty} |u(x)||x|^{N-2}=0 \,,$$ 
then by Theorem \ref{Thm:Gen_Schrodinger} it follows that $u = 0$ in $\R^N$. It is noteworthy that this gives a Landis-type result under the polynomial decay assumption on $u$.
\end{application}

\begin{application}
Another application of Theorem \ref{Thm:Gen_Schrodinger} is Theorem \ref{Thm:Schrodinger} itself. 
\begin{proof}[Proof of Theorem \ref{Thm:Schrodinger}]
Let us consider $\Om=\R^N$, $A_k=I_N$ for $k=1,2$ and $V_2\leq W=1$ in Theorem \ref{Thm:Gen_Schrodinger}. It is well known that $-\Delta$ is critical in $\R^N$ for $N=1,2$ \cite{PTT}. In these dimensions, we take $V_1= W_{\mathrm{cri}}=0$. The associated Agmon ground state of the critical operator $H_1=-\Gd$ in $\R^N$ is  $\Psi_{H_1}=1$ \cite{PTT}. For $N \geq 3$, we choose a critical Hardy-weight similar to the one in the previous application and take  
	$$V_1(x)=-W_{\mathrm{cri}}(x)=-\left(\frac{N-2}{2}\right)^2\frac{1}{|x|^2} \quad  \mbox{near infinity}.$$ 
In this case the associated Agmon ground state of the critical operator $H_1=-\Delta-W_{\text{cri}}$ is $\Psi_{H_1} \asymp |x|^{(2-N)/2}$ near infinity of $\R^N$ (see \cite[Theorem~4.12]{DFP}). On the other hand, the minimal positive Green function $\mathcal{G}_{-\Gd+1}$  of $-\Gd+1$ in $\R^N$, $N\geq 1$, is a positive solution of minimal growth at infinity of $(-\Gd+1)[\varphi]=0$ in $\R^N$,  and it behaves as 
	$$\mathcal{G}_{-\Gd+1} (x)\asymp \frac{\mathrm{e}^{-|x|}}{|x|^{(N-1)/2}} 
	\qquad \mbox{as } \ |x| \rightarrow \infty,$$
	see \cite{Norio_Shimakura}.
  Now the proof follows immediately from Theorem \ref{Thm:Gen_Schrodinger}.
\end{proof}
\end{application}
\begin{remark}\label{rem_3-2}
	Let $\mathcal{G}_{-\Gd+1}(x) := \mathcal{G}_{-\Gd+1}(x,0)$ be the minimal positive Green function  of $-\Gd+1$ in $\R^N$. Under the stronger assumption $|u(x)|=o(\mathcal{G}_{-\Gd+1}(x))$ as $|x|\to \infty$, we indicate here a simpler proof of  Theorem \ref{Thm:Schrodinger}.
	Let $\Phi$  be a positive supersolution of the equation $H[\varphi]=0$ in $\R^N$  (recall that $H\geq 0$ in $\R^N$).   Since $\Phi$ is also a positive supersolution of the equation $(-\Gd+1)[\varphi]=0$ in $\R^N$, it follows that $\Phi \geq C\mathcal{G}_{-\Gd+1}$ in $\R^N\setminus B_1(0)$ for some $C>0$. Suppose that $u \neq 0$ and without loss of generality, assume that $u^+\neq 0$. For $\vge >0$, consider the supersolution $\Phi_\vge: =\Phi - \vge u^+$ of the equation $H[\varphi]=0$ in $\R^N$, which is a nonnegative for small positive $\vge$, but it is not positive for large $\vge$ (since $u^+\neq 0$). Now, using the maximal $\vge$-argument (see the proof of Theorem~\ref{thm_element}), we will get a contradiction. Thus,  $u=0$.
\end{remark}
\begin{application}[Bounded domain] 
	Let $\Om$ be a bounded domain with a $C^1$-boundary $\partial \Om$. Assume further that  $A_k=I_N$ for $k=1,2$ and $V_2\leq W =1$ in Theorem \ref{Thm:Gen_Schrodinger}. Note that $-\Delta$ is subcritical in $\Om$ \cite{PT}. Let $\mathcal{G}_{-\Delta}$ be its minimal positive Green function. It follows from the Hopf boundary point lemma (see, \cite[Lemma 3.2]{Lamberti_Pinchover}) that $\mathcal{G}_{-\Delta} \asymp \mathrm{d}_{\Omega}$ near $\partial \Om$, where  $\mathrm{d}_{\Omega}$ is the distance function to $\partial \Gw$. Using \cite[Theorem~4.12]{DFP}, we can obtain an optimal Hardy-weight $W_{\mathrm{opt}}$ for $-\Delta$ in $\Om$ such that $W_{\mathrm{opt}}=\frac{1}{4}|\nabla \log (\mathcal{G}_{-\Delta})|^2$ outside a compact set in $\Om$. This in particular ensures that $H_1:= -\Delta - W_{\mathrm{opt}}$ is critical in $\Om$ with an Agmon ground state $\Psi_{H_1}$ that behaves like $\Psi_{H_1} \asymp \sqrt{\mathcal{G}_{-\Delta}} \asymp \sqrt{{\mathrm{d}}_{\Om}}$ as $x \rightarrow \bar{\infty} $ \cite[Theorem~4.12]{DFP}. Again, by the Hopf boundary point lemma (see, \cite[Lemma 3.2]{Lamberti_Pinchover}), it follows that the minimal positive Green function of $-\Delta+1$ satisfies
${\mathcal{G}}_{-\Gd+1} (x)\asymp {\mathrm{d}}_{\Om}(x)$ 
	as  $x \to \bar{\infty}$. Therefore, if $u \in W^{1,2}_{\loc}(\Om)$ is a solution to $(-\Delta +V)[\varphi]=0$ in $\Om$ with  $|u|=O(\sqrt{{\mathrm{d}}_{\Om}})$ in $\Om$ and 
\begin{equation}\label{eq_35}
\liminf_{x\rightarrow \bar{\infty}} \frac{|u(x)|}{{\mathrm{d}}_{\Om}(x)}=0 \,,
\end{equation}	
then by Theorem \ref{Thm:Gen_Schrodinger} we infer that $u=0$.
\end{application}
\begin{remark}[Bounded domain: an alternative approach]
In the case of a $C^1$-bounded domain $\Om$, instead of using the Liouville comparison principle, we can alternatively use  {\it criticality theory} to deduce a Landis-type theorem for the nonnegative operator $H:=-\Delta +V$ with $V\in L^\infty(\Gw)$. Let $u \in W^{1,2}_{\loc}(\Om)$ be a solution of $H[\varphi]=0$ in $\Om$ with  $|u|=O(\mathrm{d}_{\Om})$ in $\Om$. Note that the spectrum of $\tilde H$, the Friedrichs extension $H$, is discrete. Moreover,  
	$$\gl_0:=\gl_0(H,1,\Gw)=\inf \gs(\tilde H)$$
	 is its simple principal eigenvalue. Therefore, $H\geq 0$ is subcritical if and only if  $\gl_0>0$.  It follows that if $H$ is subcritical in $\Gw$, then $u=0$. On the other hand, in the critical case (i.e., $\gl_0=0$), let $\Psi_H$ be its Agmon ground state. Then  by the simplicity of $\gl_0$ we obtain that $u=c\Psi_H$, where $c\in \R$. By the Hopf boundary point lemma (see, \cite[Lemma 3.2]{Lamberti_Pinchover}) $\Psi_H \asymp \text{d}_{\Omega}$ as $x\rightarrow \bar{\infty}$. Therefore, if in addition, $u$ satisfied \eqref{eq_35}, 
then it follows that $u=0$ in $\Om$.
\end{remark}


\begin{application}[Exterior domain]\label{Exterior domain}
	 Let $\Om$ be a $C^1$-exterior domain, i.e., an unbounded domain with a nonempty compact $C^1$-boundary $\partial \Om$. Take  $A_k=I_N$ for $k=1,2$ and $V_2\leq W =1$ in Theorem \ref{Thm:Gen_Schrodinger}. Note that $-\Delta$ is subcritical in $\Om$ \cite{PT}, and let $\mathcal{G}_{-\Delta}$ be its minimal positive Green function in $\Om$. It follows from the Hopf boundary point lemma (see, \cite[Lemma 3.2]{Lamberti_Pinchover}) that $\mathcal{G}_{-\Delta} \asymp \mathrm{d}_{\Omega}$ near $\partial \Om$. Near infinity $\mathcal{G}_{-\Delta}$ behaves as a positive solution of minimal growth of $-\Delta[\varphi]=0$ in $\R^N$, i.e.,  \begin{align*}
 \mathcal{G}_{-\Delta}(x) \asymp 
 |x|^{2-N} \ \ \mbox{if} \ \ N \geq 2,  \qquad \mbox{as } |x| \rightarrow \infty.
 \end{align*}
 Using \cite[Theorem~4.12]{DFP}, we can obtain for $N\geq 3$ an optimal Hardy-weight $W_{\mathrm{opt}}$ for $-\Delta$ in $\Om$ such that $W_{\mathrm{opt}}=\frac{1}{4}|\nabla \log (\mathcal{G}_{-\Delta})|^2$ outside a compact set in $\Om$. This in particular ensures that $H_1:= -\Delta - W_{\mathrm{opt}}$ is critical in $\Om$ with an Agmon ground state $\Psi_{H_1}$ that behaves like $\Psi_{H_1} \asymp \sqrt{\mathcal{G}_{-\Delta}}$ as $x \rightarrow \bar \infty$. 
 So, $\Psi_{H_1} \asymp  \sqrt{\mathrm{d}_{\Omega}}$ near $\partial \Om$, and  $\Psi_{H_1} \asymp |x|^{(2-N)/2}$ as $|x| \rightarrow \infty $. Obviously, the minimal positive Green function of $-\Delta+1$ in $\Om$ satisfies
  	\begin{equation*}
  	\mathcal{G}_{-\Gd+1} (x)\asymp \frac{\mathrm{e}^{-|x|}}{|x|^{(N-1)/2}} 
	\qquad \mbox{as } \ |x| \rightarrow \infty  .
	\end{equation*}
	Therefore, in order to apply Theorem \ref{Thm:Gen_Schrodinger} and to conclude that the solution $u=0$, one needs only to assume that  
	\begin{align}
	 |u|=O(\sqrt{\mathrm{d}_{\Om}}) \ \ \mbox{near} \ \partial \Om\,, & \, \quad |u(x)|=O(|x|^{(2-N)/2}) \ \ \ \mbox{as} \ |x|\rightarrow \infty, \label{eq_212} \\
	& \,  \mbox{and} \ \ \liminf_{|x|\rightarrow \infty} \frac{|u(x)|}{\mathcal{G}_{-\Gd+1}(x)}=0 , \label{eq_21}
		\end{align}
	where the behaviour of $\mathcal{G}_{-\Gd+1}$ is given above. If $N=2$, using the optimal Hardy-weight construction \cite[Theorem~4.12]{DFP} one can verify that it is enough to assume $|u| = O(\sqrt{\mathrm{d}_{\Omega}})$ near $\partial \Om$, $|u(x)| = O(\sqrt{\log |x|})$ as $|x|\to \infty$  and \eqref{eq_21}.
\end{application}
\begin{remark}
In the above application, if $\Om$ is a $C^{1,\gamma}$-exterior domain in $\R^N$, $N \geq 3$ and $\gamma \in (0,1]$, then we can replace the assumption $|u(x)|=O(|x|^{(2-N)/2})$ as $|x|\rightarrow \infty$ in \eqref{eq_212} with  $|u(x)|=O(|x|^{\alpha/2})$, where $\ga=(2-N) - \sqrt{(N-2)^2-1}$.
To get this, we replace the role of $W_{\mathrm{opt}}$ in the above application by a critical Hardy-weight $W_{\mathrm{cri}}$ of $-\Delta$ in $\Om$, which is equal to  $1/(4\mathrm{d}_{\Om}^2)$ outside a compact set in $\Om$. Then, it follows from \cite[Corollary 5.5]{Lamberti_Pinchover} that the Agmon ground state $\Psi_{H_1}$ of the corresponding critical operator $H_1:=-\Delta - W_{\mathrm{cri}}$ satisfies $\Psi_{H_1} \asymp \sqrt{\mathrm{d}_{\Omega}}$ near the boundary $\partial \Om$ and $\Psi_{H_1}(x) \asymp |x|^{\ga/2}$ as $|x| \rightarrow \infty$. This proves our claim.
\end{remark}
\section{Landis-type theorem for the operator $\mathcal{L}$ and applications}\label{sec_general}
This section contains the proof of theorems~\ref{thm_element} and \ref{thm_element1} which are extensions of the Landis-type theorem for Schr\"odinger operators (Theorem~\ref{Thm:Gen_Schrodinger}) to the case of nonselfadjoint operators. The reason why we call these theorems a generalized form of Landis-type theorem, will be clear from the various applications of these theorems presented in this section. 
\subsection{Proofs of theorems~\ref{thm_element} and \ref{thm_element1}}
\begin{proof}[Proof of Theorem~\ref{thm_element}]
	Assume that $u \neq 0$.  
Without loss of generality, we may assume that $u^+\neq 0$. Then it follows from Remark \ref{Rmk:subsol} that $u^+$ is a subsolution of $\mathcal{L}[\varphi]=0$ in $\Gw$. 
By our assumption $\mathcal{L}\geq 0$ in $\Gw$, we pick a positive  supersolution  $\Phi$  of the equation $\mathcal{L}[\varphi]=0$ in $\Gw$. For $\vge>0$ consider the function $\Phi_\vge=\Phi-\vge u^{+}$.  Clearly,  $\Phi_\vge$ is a supersolution of the same underlying equation. Moreover, by our assumptions, for small enough $\vge>0$, $\Phi_\vge > 0$ in $\Gw$, and let 
	$$\vge_M= \sup\{\vge \mid  \Phi_\vge > 0\},$$
	and note that $\vge_M <\infty$. By the strong maximum principle (Remark~\ref{Rmk:subsol}), either $\Phi_{\vge_M}>0$, and this implies immediately a contradiction to the maximality of $\vge_M$, which in turn implies that either $u=0$, or $\Phi_{\vge_M}=0$.  In the latter case, it follows that $\vge_M u =\Phi$, and this implies that $u$ is a global  positive solution of $\mathcal{L}[\varphi]=0$ of minimal growth at infinity in $\Om$. So,  $\mathcal{L}$ is critical in $\Gw$, and  $u$ is an Agmon ground state of $\mathcal{L}$.       
\end{proof}
\begin{remark}
	Consider a nonnegative uniformly elliptic operator $\mathcal{L}$ in $\R^N$ of the form \eqref{Eq:Ellp_op} with bounded coefficients in $\R^N$. Then using the uniform Harnack inequality and a Harnack chain argument \cite{GT}, it follows that a global positive solution of the equation  $\mathcal{L}[\phi]=0$ in $\R^N$  cannot decay faster than a certain positive exponential $\gg$. Consequently, if $u$ is a solution of the equation $\mathcal{L}[\phi]=0$ in $\R^N$ satisfying   
	$$|u(x)|=o\left(\mathrm{e}^{-\gg|x|}\right) \qquad \mbox{as } |x| \to \infty,$$
	then using the maximal $\vge$ argument as above, it follows that $u=0$.  Hence, the Landis conjecture holds true for nonnegative uniformly elliptic operators with bounded coefficients.  We note that there are weaker assumptions that guarantee the validity of the uniform Harnack inequality. 
 	\end{remark}

\begin{proof}[Proof of Theorem~\ref{thm_element1}]
	Assume that $u \neq 0$. Without loss of generality, we may assume that $u^+\neq 0$. By Remark~\ref{Rmk:subsol}, $u^+$ is a subsolution of the equation $(\mathcal{L}+V)[\varphi]=0$ in $\Gw$. For $\vge>0$ consider the function $\Phi_\vge=\Phi-\vge u^+$, where $\Phi=\Psi_{\mathcal{L}-W}$.  Observe that  $\Phi_\vge$ is a supersolution of $(\mathcal{L}-W)[\varphi]=0$ in $\Om$. Indeed,
	\begin{align*}
	(\mathcal{L}-W) [\Phi_{\vge}]= (\mathcal{L}-W) [\Phi] - \vge (\mathcal{L}-W)[u^+] = - \vge \left( (\mathcal{L}+V)[u^+] -(V+W)u^+ \right) \geq 0 
	\end{align*}
	in $\Om$. Moreover, by our assumptions, for small enough $\vge>0$, we have $\Phi_\vge > 0$. But as the operator $\mathcal{L}-W$ is critical in $\Gw$, it admits (up to a multiplicative constant) a unique positive supersolution which is in fact, the Agmon ground state (Theorem \ref{Thm:char_cri}-${\it (4)}$). As a consequence, it follows that $u= u^+=C \Phi$, where $C$ is a positive constant. Subsequently, we reach a contradiction as $ \liminf_{x\to \bar{\infty}} |u(x)|/\Phi(x) =0$.
	Therefore, $u =0$.
\end{proof}
\begin{remark}
If one assumes that the coefficients of $\mathcal{L}$ are bounded and smooth enough and that the solution $u$ of  $(\mathcal{L}+V)[\varphi]=0$ in $\Om$ is a {\it strong solution} (i.e., $u \in W^{2,N}_{\loc}(\Om)$ and satisfies the equation a.e. in $\Om$), then one can give an alternative proof of Theorem \ref{thm_element1} by using a variant of Liouville comparison principle established in  \cite[Theorem 2.2]{ABG}. 

\end{remark}
Using \cite[Corollary 5.2 and Theorem 4.12]{DFP} we obtain a corollary of Theorem~\ref{thm_element1}.

\begin{corollary}\label{cor_opt}
Let $N \geq 2$ and $\mathcal{L}$ is an elliptic operator of the form \eqref{Eq:Ellp_op} which is subcritical in $\Gw$  with minimal positive Green function $\mathcal{G}_{\mathcal{L}}$ with singularity at $y\in \Gw$.

Suppose that there exists a positive solution $w \in W^{1,2}_{\loc}(\Gw)$ of $\mathcal{L}[\varphi]=0$ in $\Gw$ satisfying the Ancona condition \cite{DFP} $\lim_{x\rightarrow \bar\infty}(\mathcal{G}_{\mathcal{L}}/w)=0$, and that $\mathcal{L}\left[\sqrt{\mathcal{G}_{\mathcal{L}}w}\,\right]\in L^q_\loc(\Gw)$, $q>N/2$.

	Let $u \in W^{1,2}_{\loc}(\Gw)$ be a  solution of the equation $\mathcal{L}[\varphi]=0$ in $\Gw$ satisfying   
	\begin{equation}\label{eq_73}
	|u(x)| =O\left(\sqrt{\mathcal{G}_{\mathcal{L}}(x)w(x)}\right) \  \ \mbox{ as } x \to\bar\infty \ \mbox{ and} \ \ \  \liminf_{x\to \bar\infty} \frac{|u(x)|}{\sqrt{\mathcal{G}_{\mathcal{L}}(x)w(x)}} =0 \,.
	\end{equation}
	Then $u = 0$.
\end{corollary}
\begin{proof}
	Using \cite[Theorem 4.12]{DFP}, it follows that  
		$$W=\frac{\mathcal{L}\left[\sqrt{\mathcal{G}_{\mathcal{L}}w}\,\right]}
		{\sqrt{\mathcal{G}_{\mathcal{L}}w}} \ \ \mbox{ in } \ \Om\setminus \{y\}$$ 
 is an optimal Hardy-weight of $\mathcal{L}$ in $\Gw\setminus \{y\}$. In particular, $\mathcal{L} -W$ is critical in $\Gw \setminus \{y\}$ with the Agmon ground state $\sqrt{\mathcal{G}_{\mathcal{L}}w}$. Moreover, \eqref{eq_73} is satisfied also in $\Gw \setminus \{y\}$. Hence, Theorem~\ref{thm_element1} with $V=0$ in $\Gw \setminus \{y\}$, this implies that $u = 0$ in $\Gw \setminus \{y\}$, and by continuity, $u = 0$ in $\Gw$. 
\end{proof}

\begin{remark} \label{Rmk_compare}
Theorem \ref{thm_element1} is stronger than Theorem \ref{thm_element}.
Indeed, by taking $V= 0$ in Theorem \ref{thm_element1}, one can see that the assumption $V + W \geq 0$ is automatically satisfied for a critical Hardy-weight $W \in L^{q}_{\loc}(\Om)$ with $q>N/2$. Observe that the Agmon ground state $\Psi_{\mathcal{L}-W}$ of $\mathcal{L}-W$ in $\Om$ and the positive minimal Green function $\mathcal{G}_{\mathcal{L}}$ of $\mathcal{L}$ in $\Om$  satisfy $\Psi_{\mathcal{L}-W} \geq C \mathcal{G}_{\mathcal{L}}$ near $\bar\infty$. Consequently, Theorem \ref{thm_element1} proves the Landis-type theorem under a slower decay assumption on the solution $u$ than Theorem \ref{thm_element}. However,  in many cases, it is easier to find a lower bound for $\mathcal{G}_{\mathcal{L}}$, while finding an explicit nontrivial critical Hardy-weight $W$  with an Agmon ground state $\Psi_{\mathcal{L}-W}$  might be rather difficult. Moreover, the existence of an optimal Hardy-weight via the construction in \cite[Theorem 4.12]{DFP} is guaranteed only in the {\it{quasi-symmetric}} case \cite{Ancona1}. Therefore, we prefer keeping both theorems.
\end{remark}
\subsection{Applications of theorems~\ref{thm_element} and \ref{thm_element1}}
As we mentioned in the introduction, in order to obtain a Landis-type result for a nonnegative operator $\mathcal{L}$ using theorems~\ref{thm_element} and \ref{thm_element1}, it is enough to find lower bounds for $G_{\mathcal{L}}$ and $\Psi_{\mathcal{L}-W}$, respectively. In this subsection, we provide explicit estimates of these functions for some operators.
\begin{application}[The periodic case]\label{prop_periodic} Consider an elliptic operator $\mathcal{L} \geq 0$  on $\R^N$ in the divergence form as in \eqref{Eq:Ellp_op}. Assume that the  coefficients of $\mathcal{L}$ are $\mathbb{Z}^{N}$-periodic, that is,  
	$$a_{ij}(x+z)= a_{ij}(x), \ \tilde{b}_{i}(x+z)=\tilde{b}_{i}(x),\  b_{i}(x+z)=b_{i}(x), \ c(x+z)=c(x) \ \ \forall x \in \R^N \mbox{ and } \forall z \in \mathbb{Z}^{N},$$ 
	and all the other assumptions of Theorem~\ref{thm_element} are satisfied. 
 Now consider the set
\begin{multline*}
\Gg:=\left\{\beta \in \R^N \mid \exists \, \psi \in W^{1,2}_{\loc}(\R^N) \ \mbox{such that} \ \mathcal{L}[\psi] =  0  \mbox{ in }  \R^N, \right.\\
 \left. \mbox{and }  \psi=\mathrm{e}^{-\beta \cdot x} w_\gb  \mbox{ for some positive periodic } w_\gb \right\} .
 \end{multline*}
It has been shown by Agmon \cite{Ag} that $\Gg\neq \emptyset$ if $\gl_0:=\gl_0(\mathcal{L},1,\R^N)\geq 0$. Moreover, either $\Gg =\{\beta_0\}$ is singleton and this happens if and only if $\gl_0=0$, or $\Gg$ is the boundary of an $N$-dimensional strictly convex compact set $K$ with
smooth boundary $\partial K=\Gg$, and this happens if and only if $\gl_0>0$. In addition, if $\mathcal L$ is symmetric, then $\Gg=-\Gg$. By the compactness of $\Gg$, for each $s \in \mathbb{S}^{N-1}$ there exists $\beta(s) \in \Gg$ such that
$$ \sup_{\beta \in \Gg} \beta \!\cdot\! s= \beta(s) \!\cdot\! s.$$

Observe that if $\gl_0>0$, then $\mathcal{L}$ is subcritical. While, if $\gl_0=0$, then $\mathcal{L}$ is subcritical if and only if $N \geq 3$ \cite[Theorem 6.2]{Murata}. In this discussion, we consider $N \geq 3$, and hence, $\mathcal{L}$ is subcritical.

Suppose that $\gl_0>0$. Murata and Tsuchida \cite[Theorem 1.1]{Murata} showed that in this case the  minimal positive Green function $\mathcal{G}_{\mathcal{L}}$ of $\mathcal{L}$ in $\R^N$ satisfies  
\begin{align} \label{Eq:Green_est}
\mathcal{G}_{\mathcal{L}} (x) \asymp 
\frac{\mathrm{e}^{-\beta({x}/{|x|})\, \cdot\, x}}{|x|^{\frac{N-1}{2}}} \quad  \mbox{ as } \ |x| \rightarrow \infty \,.
\end{align}
	On the other hand, if $\gl_0=0$, then by  \cite[Theorem 1.3]{Murata}
	$$\mathcal{G}_{{\mathcal{L}}}(x)\asymp \mathrm{e}^{-\gb_0\cdot x}|x|^{2-N} 
	\quad \mbox{as} \ |x| \rightarrow \infty.$$ 
Clearly, a positive solution $G_\mathcal{L}$ of minimal growth at infinity of the equation $\mathcal{L}[\varphi]=0$ in $\R^N$ behaves like $G_\mathcal{L} \asymp \mathcal{G}_{\mathcal{L}}$ near infinity in $\R^N$. Thus, we obtain the following result as a particular case of Theorem~\ref{thm_element}:

Suppose that $N\geq 3$, and $\mathcal{L} \geq 0$ in $\R^N$ is an operator of the form \eqref{Eq:Ellp_op} with $\mathbb{Z}^{N}$-periodic coefficients. 
Let $u \in W^{1,2}_{\loc}(\R^N)$ be a solution of the equation $\mathcal{L}[\varphi]=0$ in $\R^N$ satisfying   
\begin{equation*}\label{eq_32}	
\begin{cases}
|u(x)| =O(|x|^{\frac{1-N}{2}}\mathrm{e}^{-\beta({x}/{|x|}) \cdot x})   \mbox{ as } |x| \to\infty & \mbox{if} \ \gl_{0}>0,\\[4mm]
|u(x)| =O(\mathrm{e}^{-\gb_{0}\cdot x}|x|^{2-N})   \mbox{ as } |x| \!\to\!\infty
& \mbox{if} \  \gl_{0}=0,
\end{cases}
\end{equation*}
where $\gl_{0}=\gl_0(\mathcal{L},1,\R^N)$. Then $u = 0$.  
\end{application}
We note that using  Liouville-type theorems for periodic operators proved in \cite{Kuchment}, one gets a Landis-type results for solutions of polynomial growth.    
\begin{remark}
For a (non-sharp) two-sided estimate of the positive minimal Green function of the operator 
\begin{align*} 
\mathcal{L}[\varphi]:=-\Delta \varphi + a(x) \cdot \nabla \varphi + \la \varphi \qquad \mbox{in } \R^N, 
\end{align*}
where $a\in L^\infty(\R^N)$ and $\la >0$, see \cite[Theorem 4.1]{Hill}.
\end{remark}
\begin{remark}
	One can use sufficient conditions for the equivalence of Green functions of two subcritical operators (see for example \cite{Ancona,Murata1,P_JDE1,P_DIEQ} and references therein), to obtain Landis-type theorems for {\it{small perturbations}} of an operator $\mathcal{L}$ with a known lower bound for its Green function $\mathcal{G}_{\mathcal{L}}$. 
\end{remark}
Now we present an explicit application of Corollary~\ref{cor_opt}.

\begin{application}\label{app-4_3}
	Take $\mathcal{L}=-\Delta +1$ in $\R^N$. Recall that the minimal Green function $\mathcal{G}_{\mathcal{L}}$ of $\mathcal{L}$ satisfies $$\mathcal{G}_{\mathcal{L}}(x) \asymp \frac{{\mathrm{e}}^{-|x|}}{|x|^{(N-1)/2}} \ \ \mbox{as} \ |x| \rightarrow \infty \,.$$
	It is known that there exists a positive solution $w$ of $\mathcal{L}[\varphi]=0$ in $\R^N$ such that
	$$w(x) \asymp \frac{{\mathrm{e}}^{|x|}}{|x|^{(N-1)/2}} \ \ \mbox{as} \ |x| \rightarrow \infty \,,$$
	see Appendix. 
	Clearly, $\mathcal{G}_{\mathcal{L}}$ and $w$ satisfy the Ancona condition i.e., $\lim_{|x|\rightarrow \infty} \mathcal{G}_{\mathcal{L}}/w=0.$ Subsequently, using Corollary~\ref{cor_opt}, it follows that if $u \in W^{1,2}_{\loc}(\Gw)$ is a  solution of the equation $\mathcal{L}[\varphi]=0$ in $\R^N$ satisfying   
	\begin{equation}\label{eq_137}		
	|u(x)| =O\left(|x|^{\frac{(1-N)}{2}}\right) \  \ \mbox{ as } |x| \to \infty, \ \mbox{ and} \ \ \  \liminf_{|x| \to \infty}   |u(x)||x|^{\frac{(N-1)}{2}}  =0,
	\end{equation}
	then $u = 0$.
\end{application}
  \begin{remark}
  	Note that Theorem \ref{Thm:Schrodinger}  also gives a Landis-type theorem for the particular case $\mathcal{L}=-\Delta +1$ in $\R^N$, but under the assumptions \eqref{Eq:cond_Thm:Schrodinger} which are  stronger than \eqref{eq_137}.
  \end{remark} 
\section{Quasilinear Landis-type theorem}\label{Sec-Quasi}
Let $A_k: \Om \rightarrow \R^{N^2}$, $k=1,2$ be two matrices as in Theorem \ref{Thm:Gen_Schrodinger}. In addition, assume that the entries of $A_k$ are H\"older continuous in $\Om$. For $1<p<\infty$ and $V_k \in L^{\infty}_{\loc}(\Om)$, consider two quasilinear 
Schr\"odinger type operators:
\begin{align} \label{Eq:quas_Sch_op}
\mathcal{Q}_k := -\Delta_{p,A_k} + V_k \mathcal{I}_p \,, \quad k=1,2,
\end{align}
where $\mathcal{I}_p(\varphi)=|\varphi|^{p-2}\varphi$ and $\Delta_{p,A}$ denotes the so-called $(p,A)$-Laplacian which is defined as 
$$\Delta_{p,A}(\varphi)={\mathrm{div}}(|\nabla \varphi|_{A(x)}^{p-2} A(x) \nabla \varphi) \,,$$
where $A=(a_{ij})\!:\! \Om \rightarrow \R^{N^2}$, and $|\xi|_{A(x)}^2 \! := \! A(x) \xi \!\cdot\! \xi$ for all $x \!\in\! \Om$ and $\xi \in \R^N$. Assume that  for $k=1,2$, $\mathcal{Q}_k$ is nonnegative, i.e., it admits a (weak) positive supersolution
$v_k\in W^{1,p}_{\loc}(\Om)$ of $\mathcal{Q}_k[\phi]=0$ in $\Gw$. The criticality theory has been developed for such operators in \cite{PP}. In particular, the notions of positive solution of minimal growth, minimal positive Green function, Agmon ground state, etc., have been extended to the quasilinear setting, see \cite{PP}. Using this criticality theory for the operators of the above form, the following quasilinear analogue of the linear Liouville comparison theorem (Theorem~\ref{LC}) is proved in  \cite{PP,PTT}.
\begin{theorem}[Quasilinear Liouville comparison principle] \label{Thm:quas_LC}
Let $\Omega$ be a domain in $\R^N$, $N\geq 1$, and consider two nonnegative quasilinear Schr\"odinger operators $\mathcal{Q}_k$ as in \eqref{Eq:quas_Sch_op}. Assume that the following assumptions hold true.
	\begin{itemize}
		\item[(i)] The operator  $\mathcal{Q}_1$ is critical in $\Omega$ with its Agmon ground state $\Psi_{\mathcal{Q}_1}$. 
		
		\item[(ii)] There exists $\Phi_2 \in W^{1,p}_{\mathrm{loc}}(\Omega)$ with
		$\Phi_2^+\neq 0$ such that $\mathcal{Q}_2[\Phi_2] \leq 0$ in $\Omega$. 
		
		\item[(iii)] There exists a positive constant C such that the following matrix inequality holds
		\begin{equation}\label{psialephia1}
		(\Phi_2^+)^2(x) A_2(x)\leq C\Psi_{\mathcal{Q}_1}^2(x) A_1(x) \quad  \mbox{ a.e.
			in } \Omega.
		\end{equation}
		\item[(iv)] There exists $C>0$ such that $$|\nabla \Phi_2(x)|_{A_2(x)}^{p-2} \leq C |\nabla \Psi_{\mathcal{Q}_1}(x)|_{A_1(x)}^{p-2} \qquad \mbox{for a.e. } x \in \Om \cap (\Phi_2 >0).$$
	\end{itemize}
	Then the operator $\mathcal{Q}_2$ is critical in $\Omega$, and $\Phi_2$ is its Agmon
	ground state.  
\end{theorem} 
Theorem~\ref{Thm:quas_LC} extends Theorem \ref{LC} to the quasilinear case as the extra assumption $(iv)$ in the above theorem boils down to a trivial one when $p=2$. Having this Liouville comparison theorem for the quasilinear operators, we can repeat the same arguments as in Theorem \ref{Thm:Gen_Schrodinger} to prove the following Landis-type theorem for quasilinear  Schr\"odinger-type equations.  
\begin{theorem}\label{Thm:Gen_quasi_Schrodinger} 
Let $\Om$ be a domain in $\R^N$, $N\geq 1$ and	 consider two nonnegative quasilinear  Schr\"odinger-type operators $\mathcal{Q}_k$ on $\Om$ as in \eqref{Eq:quas_Sch_op}.  Let $W \in L^{\infty}_{\loc}(\Om)$ satisfies $W \geq V_2$ outside a compact set in $\Om$.
  
	Suppose that $\mathcal{Q}_1$ is critical in $\Gw$ with an Agmon  ground state $\Psi_{\mathcal{Q}_1}$, and  let $u \in W^{1,p}_{\loc}(\Om)$ be a solution of the equation $\mathcal{Q}_2 [\varphi]=0$ in $\Om$  satisfying   
\begin{align}	\label{assumption_1}
\! \!\!\! \! u(x)^2A_2(x)\leq C\Psi_{\mathcal{Q}_1}(x)^2A_1(x) & \  \mbox{a.e. in } \Gw, \ \ \ |\nabla u(x)|_{A_2(x)}^{p-2} \leq C |\nabla \Psi_{\mathcal{Q}_1}(x)|_{A_1(x)}^{p-2}  \ \,  \mbox{a.e. in } \Gw,   \\[2mm]
&  \ \mbox{ and } \ \ \liminf_{x\to \bar{\infty}} \frac{|u(x)|}{G_{-\Delta_{p,A_2}+W}(x)} =0 \label{assumption_2} \,,
\end{align}
where $G_{-\Delta_{p,A_2}+W}$ is a positive solution of minimal growth of $(-\Delta_{p,A_2}+ W\mathcal{I}_p)[\varphi]=0$ in a neighbourhood of infinity in $\Om$.
Then $u = 0$.   
\end{theorem}
\begin{proof}
It follows as the proof of Theorem \ref{Thm:Gen_Schrodinger}, and hence it is omitted.
\end{proof}
Next we provide some explicit examples where Theorem \ref{Thm:Gen_quasi_Schrodinger} can be applied.
\begin{application}\label{app_53}
Assume that in the above theorem  $\Om=\R^N$, $A_k=I_N$ for $k=1,2$, and $V_2\leq W=0$ outside a compact set. A positive solution  $G_{-\Delta_p+W}$ of the equation $(-\Delta_p+W\mathcal{I}_p)[\varphi]=0$ in $\Om$ of minimal growth at infinity satisfies \cite[Example~1.7]{PTT}  
 \begin{align*}
 G_{-\Delta_p+W}(x)=G_{-\Delta_p}(x) \asymp \begin{cases} |x|^{\frac{p-N}{p-1}} & \mbox{ if }\   p<N, \\
1 & \mbox{ if }\  p\geq N ,
 \end{cases} \quad  \ \mbox{as} \ \ |x| \to \infty \,.
 \end{align*}
In dimension $N=1$, the $p$-Laplacian $-\Delta_p$ is critical with the Agmon ground state $\Psi_{-\Delta_p}=1$ \cite{PTT}. Hence, if $u \in W^{1,p}_{\loc}(\R)$ is a bounded solution of $(-\Delta_p +V_2\mathcal{I}_p)[\varphi]=0$ in $\R$ satisfying $\liminf_{|x|\rightarrow \infty} |u(x)|=0$, then for $p \leq 2$, Theorem~\ref{Thm:Gen_quasi_Schrodinger} (by taking $\mathcal{Q}_1=-\Delta_p$ and $\mathcal{Q}_2=-\Delta_p+V_2 \mathcal{I}_p$) ensures that $u = 0$.

In the above case, if $p \geq 2$ then the assertion of Theorem \ref{Thm:Gen_quasi_Schrodinger} reduces to a trivial one due to the second assumption in \eqref{assumption_1}. Note that the same happens in the case of $2 \leq N \leq p$. 


Now consider the case $N \geq 2$ and $2\leq p<N$. In this case, it is known that $\mathcal{Q}_1 := -\Delta_p - W_{\mathrm{opt}} \mathcal{I}_p$ is critical in $\R^N \setminus \{0\}$ with an Agmon ground state $\Psi_{\mathcal{Q}_1}=|x|^{\frac{p-N}{p}}$, where 
$$W_{\mathrm{opt}}:=\left|\frac{p-N}{p} \right|^p \frac{1}{|x|^p}$$ 
in $\R^N \setminus \{0\}$ \cite{DP}.  Hence, if $u \in W^{1,p}_{\loc}(\R^N)$ is a  solution of $(-\Delta_p +V_2)[\varphi]=0$ in $\R^N$ satisfying 
$$|u|=O(|x|^{\frac{p-N}{p}}) , \ \ \ |\nabla u| \leq C |x|^{-\frac{N}{p}} \ \ \mbox{ a.e. in } \R^N, \ \mbox{ and } \ \ \liminf_{|x|\rightarrow \bar\infty} |u(x)||x|^{\frac{N-p}{p}}=0, $$ then, by taking $\Om=\R^N \setminus \{0\}$, $\mathcal{Q}_1=-\Delta_p - W_{\mathrm{opt}}\mathcal{I}_p$ and $\mathcal{Q}_2=-\Delta_p+V_2\mathcal{I}_p$ in Theorem~\ref{Thm:Gen_quasi_Schrodinger},  we infer that $u = 0$ in $\R^N \setminus \{0\}$ and hence $u = 0$ in $\R^N $.
\end{application}  

We have observed in the above application that the assertion of Theorem \ref{Thm:Gen_quasi_Schrodinger} trivially follows from the second assumption in \eqref{assumption_1} if $1 \leq N \leq p$. In the following proposition, it is shown how to bypass the second assumption in \eqref{assumption_1} of Theorem \ref{Thm:Gen_quasi_Schrodinger}.

\begin{proposition}
Let $1 \leq N \leq p$ and $V \in L^{\infty}_{\loc}(\R^N)$ be such that $V \geq 0$ in $\R^N$. If $u \in W^{1,p}_{\loc}(\R^N)$ is a bounded solution of $(-\Delta_p+V\mathcal{I}_p)[\varphi]=0$ in $\R^N$ satisfying $u =o(1)$, then $u =0$.
\end{proposition}  
\begin{proof}
Assume that $u \neq 0$ in $\R^N$. Without loss of generality, assume that $u^+ \neq 0$. Then it follows from \cite[Lemma 3.4]{PP} that $u^+$ is a subsolution of $(-\Delta_p+V \mathcal{I}_p)[\varphi]=0$ in $\R^N$.  Note that $\mathcal{Q}_1:=-\Delta_p$ is critical with Agmon ground state $\Psi_{\mathcal{Q}_1}=1$ in $\R^N$. For $\vge >0$, consider $\Phi_\vge: =1 - \vge u^{+}$. By our assumption $\Phi_\vge > 0$ in $\R^N$ for small enough $\vge$. Observe that $\Phi_{\vge}$ is a positive supersolution of $-\Delta_p[\varphi]=0$ in $\R^N$. Indeed,
$$ -\Delta_p[\Phi_{\vge}]=(-\vge^{p-1})(-\Delta_p[u^+])= (-\vge^{p-1})\left((-\Delta_p+V \mathcal{I}_p)[u^+]-V\mathcal{I}_p[u^+] \right) \geq 0 .$$
However, since $-\Delta_p$ is critical, there exists a unique positive supersolution which is the Agmon ground state $\Psi_{\mathcal{Q}_1}=1$ (up to a positive multiplicative constant) \cite[Theorem 4.15]{PP}. This ensures that $u=u^+=C$ for some positive constant $C$, which contradicts our assumption $u=o(1)$. Therefore, $u =0$ in $\R^N$. 
\end{proof} 

\begin{application}
In Theorem \ref{Thm:Gen_quasi_Schrodinger}, we consider $\Om=\R^N$, $A_k=I_N$ for $k=1,2$, and $V_2\leq W=1$ outside a compact set. Let $2 \leq p <N$. Due to assumption \eqref{assumption_2}, we are required getting an explicit estimate near infinity of the positive solution $G_{-\Delta_p + 1}$ of $(-\Delta_p + \mathcal{I}_p)[\varphi] =0$ of minimal growth near infinity in $\R^N$. Towards this, we take the function $$h(x):=|x|^{-\nu}{\mathrm{e}}^{-\mu |x|},$$ where $\mu, \nu >0$ satisfy $N-1-2\nu(p-1) \leq 0$ and $\nu (N-p- \nu (p-1)) \leq 0$. Then, it follows from \cite[Lemma 5.8]{Vitaly} that $h$ is a subsolution of $(-\Delta_p + \mu^p(p-1)\mathcal{I}_p)[\varphi] =0$ in $\R^N$. Further assume that $\nu > \frac{N-1}{p}$ and $\mu^p(p-1) \geq 1$. Then, using \cite[Theorem B.1 and Lemma B.2]{Vitaly} we assure that there exists a smooth compact set $K$ and $C_K>0$ such that $$G_{-\Delta_p + 1} \geq C_K h \ \ \ \mbox{in} \ \  \R^N \setminus K \,.$$

Therefore, as in Application~\ref{app_53}, if $u \in W^{1,p}_{\loc}(\R^N)$ is a  solution of $(-\Delta_p +V_2 \mathcal{I}_p)[\varphi]=0$ in $\R^N$ satisfying 
$$|u|=O(|x|^{\frac{p-N}{p}}), \  \ \ \ |\nabla u| \leq C |x|^{-\frac{N}{p}} \ \ \mbox{a.e. in }  \ \R^N, \ \ \mbox{ and } \ \  \liminf_{|x|\rightarrow \infty} \frac{|u(x)|}{h(x)}=0,$$ 
then, by taking $\mathcal{Q}_1=-\Delta_p - W_{\mathrm{opt}}\mathcal{I}_p$ and $\mathcal{Q}_2=-\Delta_p+V_2 \mathcal{I}_p$, Theorem~\ref{Thm:Gen_quasi_Schrodinger}  ensures that $u = 0$.
\end{application}
\section{Concluding remarks}\label{sec-6}
Let $\mathcal{L}$ be an elliptic operator of the form \eqref{Eq:Ellp_op}. Up to this point, we have examined 
the Landis conjecture under the assumption that the underlying operator is nonnegative in $\Om$. Now let us assume that the operator $\mathcal{L}$ is nonnegative outside a compact set $K \subset \Gw$. Under this condition, we can establish Landis-type result by exploiting a similar idea as in Application \ref{Exterior domain} for exterior domains, and an additional assumption that $\mathcal{L}$ satisfies the {\em unique continuation property} (in short, UCP) in $\Om$. {We say that  $\mathcal{L}$ satisfies UCP in $\Om$ if 
any solution $v$ of $\mathcal{L}[\varphi]=0$ in $\Om$ which vanishes on an open ball in $\Om $ vanishes in $\Gw$. For sufficient conditions for the UCP, see \cite{H}.}
\begin{theorem}
Let $\mathcal{L}$ be an elliptic operator of the form \eqref{Eq:Ellp_op} which is nonnegative in $\Om \setminus K$ for some $K \Subset \Om$, and satisfies the UCP in $\Om$.

Let $u \in W^{1,2}_{\loc}(\Om)$ be a solution of $\mathcal{L}[\varphi]=0$ in $\Om$ such that $u$ has constant sign in a neighbourhood of $\partial K$, and   
$$|u|=O(G_{\mathcal{L}}) \ \mbox{ as} \ \ x\to \bar{\infty}_\Gw \ \  \mbox{ and }  \ \ \liminf_{x\to \bar{\infty}_\Gw} \frac{|u(x)|}{G_{\mathcal{L}}^\Gw(x)} =0 \,,$$
where where $G_{\mathcal{L}}^\Gw$ is a positive solution of the equation
$\mathcal{L}[\varphi]=0$ of minimal growth at infinity in $\Gw$, and $\bar \infty_\Gw$ is the ideal point in the one-point compactification of $\Gw$.
Then $u=0$.
\end{theorem}
\begin{proof}
{Without loss of generality, let $u \leq 0$ in a neighbourhood of $\partial K$. If $u\leq 0$ in $\Om \setminus K$, then our assumption $\liminf_{x\to \bar{\infty}_\Gw} \frac{|u(x)|}{G_{\mathcal{L}}^\Gw(x)} =0$ implies that $u=0$ in $\Gw$. On the other hand, suppose that $u^+\neq 0$ in $\Om \setminus K$. Let $G_{\mathcal{L}}^{\Om \setminus K}$ be a positive solution of minimal growth at infinity in $\Om \setminus K$. Obviously, $G_{\mathcal{L}}^{\Om \setminus K} \leq C G_{\mathcal{L}}^{\Om}$ in $\Om \setminus K$ for some $C>0$ and  $G_{\mathcal{L}}^{\Om \setminus K} \asymp  G_{\mathcal{L}}^{\Om}$ as $x \to \bar\infty_\Gw$. Having observed this, one can repeat the arguments of our proof of Theorem \ref{thm_element} to infer that $u^+=u =0$ in $\Om \setminus K$. Hence, the UCP implies that $u=0$ in $\Om$.}
\end{proof}
\begin{remark}
We note that in \cite[Theorem 1.2]{Sirakov}, to establish the Landis-type theorem for a nonnegative operator $\mathcal{L}$ on an exterior domain, the authors assumed that the solution $u$ has constant sign on the boundary. In this view, while dealing with an operator $\mathcal{L}$ which is nonnegative on $\Om \setminus K$ for some $K \Subset \Om$, it seems to be natural to assume that $u$ has constant sign in a neighbourhood of the boundary of $K$.  {However, this assumption is indeed nontrivial.}
\end{remark}

Now we consider the  operator $H$ as given in \eqref{Eq:Def_Schordinger}. The assumption that $H$ is nonnegative outside a compact set $K\subset \Gw$ implies (see for example \cite[Remark~3.1]{P} and references therein) that $\gl_\infty$, the bottom of the essential spectrum of $H$, is nonnegative.  Moreover, in this case, the {\it{Morse index}} of $H$ is finite, i.e., $H$ admits at most finitely many negative eigenvalues (see \cite{D} and the references therein).  

Suppose now that $H\not \geq 0$ in $\Gw \setminus K$ for all $K$ compact set in $\Gw$.  It follows from \cite[Remark~3.1]{P} that $\gl_\infty\leq 0$. So, we are left with two cases: either
	
	(a)  $\gl_\infty  <0$, or

(b) $\gl_\infty=0$ and by \cite{D}, the Morse index of $H$ is not finite ($H$ admits infinitely many negative eigenvalues). 
\begin{remark}
Consider the cases (a) or (b) and $\Gw=\R^N$.    If $|V|\leq 1$, then $-1\leq \gl_0:
=\inf \gs(H)\leq 0$, where $\gs(H)$ denotes the spectrum of $H$. A priori,  $\gl_0$ (and also $\gl_\infty$) might be negative and even take the value  $-1$ (take for example, $V=-1$). On the other hand, in the aforementioned two cases, $\gl_\infty \leq 0$, and roughly speaking, for regularly behaved potential $V$, one does not expect to find localized eigenfunctions, i.e. exponentially decaying solutions embedded in the essential spectrum \cite{Stolz}. 
 
Let us discuss briefly case (a). It is known that under certain conditions the essential spectrum of $H$ satisfies  $\gs_{\mathrm{ess}}(H)= [\lambda_\infty, \infty)$. A sufficient condition for this is that the potential $V$ approaches a constant  $c \in [-1,0]$ as $|x|\to \infty$, and in this case $\lambda_\infty=c$.  Suppose that $\sigma_\mathrm{ess}(H)= [\lambda_\infty, \infty)$, where $\lambda_\infty<0$, then there is a huge literature on the absence of eigenvalues inside the essential spectrum (non-existence of embedded eigenvalues). Again for this to hold one has to assume that $V$ satisfies some mild regularity assumptions and certain behaviour at infinity. Under this situation, the non-existence of embedded eigenvalue ensures in particular that there is no exponentially decaying solution to $H[\varphi]=0$ in $\R^N$.

 The situation when we have embedded eigenvalues needs further investigation. Also, the case (b) is left for future studies.    
\end{remark}
 \begin{example}
 	Let $V \in L^{q}_{\loc}(\R^N)$ with $q>N/2$, such that  $0\leq -V(x)\leq  C|x|^{-2+\vge} $ in $B_R^c$ with $0<\vge<2$ and $R>0$, then $\gl_\infty=0$, $\sigma_\mathrm{ess}(H)=[0,\infty)$,  and  $H$ has infinite Morse index \cite[Theorem XIII.6]{RS}. Recall that by the Cwikel-Lieb-Rozenblum (CLR) bound on the number of negative eigenvalues of a Schr\"odinger operator \cite[Theorem XIII.12]{RS}, if $H$ has  infinite Morse index and $N\geq 3$, then $V_- \not \in L^{N/2}(\R^N)$.  
 \end{example}
\appendix 
\section{Positive radial solution of $(-\Gd+1)\gf=0$ in $\R^N$}\label{Appendix} 
We show that the equation $(-\Delta +1)\gf =0$ in $\R^N$ admits a positive solution $w$ satisfying
\begin{align} \label{Eq:est10}
w \asymp \frac{{\mathrm{e}}^{|x|}}{|x|^{(N-1)/2}} \qquad \mbox{as } \ |x| \rightarrow \infty \,.
\end{align}
Indeed, a radial solution $v=v(r)$ of the above equation, should satisfy the equation
$$\frac{\mathrm{d}^2 \varphi}{\mathrm{d}r^2}+\frac{N-1}{r} \frac{\mathrm{d}\varphi}{\mathrm{d}r}-\varphi=0 \ \ \mbox{for} \ r > 0 \,.$$
It follows that a {\em regular} solution at $r=0$ of the above equation is of the form $v(r)=r^{(2-N)/2}I_{(N-2)/2}(r)$, where $I_{(N-2)/2}(r)$ is the {\it modified Bessel function} that solves
$$\frac{\mathrm{d}^2 \varphi}{\mathrm{d}r^2}+ \frac{1}{r} \frac{\mathrm{d}\varphi}{\mathrm{d}r}-\left(1+\frac{(N-2)^2}{r^2} \right)\varphi=0 \qquad \mbox{for } \ r > 0 \,.$$
Therefore, $w(x):=v(|x|)$ solves the equation $-\Delta +1 =0$ in $\R^N$. 
Now the estimate  \eqref{Eq:est10} of $w$ follows from the asymptotic of $I_{(N-2)/2}(r)\sim c_N\mathrm{e}^r/\sqrt{r}$ as $r \rightarrow \infty$.}

Similarly, the positive minimal Green function with singularity at the origin is radially symmetric and is given by $\mathcal{G}_{-\Delta+1}(x)=C_N|x|^{(2-N)/2}H^1_{(N-2)/2}(|x|)$, where $H^1_\nu$ is the Hankel function of the first kind. For a detailed discussion on this equation, we refer to \cite[p.24]{Norio_Shimakura}.
\begin{center}
	{\bf Acknowledgements}
\end{center}
U.D. and Y.P.  acknowledge the support of the Israel Science Foundation (grant 637/19) founded by the
Israel Academy of Sciences and Humanities. U.D. is also supported in part by a fellowship from the Lady Davis Foundation.  Additionally, the authors acknowledge the DFG for supporting a stay of U.D. at the University of Potsdam.

\end{document}